\documentclass[12pt]{amsart} 

\usepackage[utf8]{inputenc}
\usepackage{amsmath}
\usepackage{amsthm}
\usepackage{amsfonts}
\usepackage{amssymb}
\usepackage{latexsym} 
\usepackage[margin=1in]{geometry}
\usepackage{enumerate} 
\usepackage[colorlinks=true]{hyperref}
\usepackage{mathrsfs}
\usepackage{tikz}
\usepackage[parfill]{parskip}
\usetikzlibrary{matrix,arrows,decorations.pathmorphing}
\usepackage{tikz-cd}
\usepackage[all]{xy}
\usepackage{titlesec} 
\usepackage{setspace}
\usepackage{xcolor}

\DeclareMathOperator{\Ext}{\mathrm{Ext}}

\DeclareMathOperator{\ann}{\mathrm{ann}}
\DeclareMathOperator\ca{\mathrm{ca}}
\DeclareMathOperator{\m}{\mathfrak{m}}
\DeclareMathOperator{\module}{\mathrm{mod}}
\DeclareMathOperator{\MCM}{\mathrm{MCM}}
\DeclareMathOperator{\sMCM}{\underline{\MCM}}
\DeclareMathOperator{\spec}{\mathsf{Spec}}
\DeclareMathOperator{\sing}{\mathrm{Sing}}

\newtheorem{theorem}{Theorem}[section]
\newtheorem*{theorem*}{Theorem}
\newtheorem{lemma}[theorem]{Lemma}
\newtheorem*{lemma*}{Lemma}
\newtheorem{proposition}[theorem]{Proposition}
\newtheorem*{proposition*}{Proposition}
\newtheorem{corollary}[theorem]{Corollary}
\newtheorem*{corollary*}{Corollary}

\theoremstyle{definition}
\newtheorem{definition}[theorem]{Definition}
\newtheorem*{definition*}{Definition}
\newtheorem{example}[theorem]{Example}
\newtheorem{remark}[theorem]{Remark}

\newtheorem{convention}[theorem]{Convention}
\numberwithin{equation}{theorem}

\newtheorem{introtheorem}{Theorem}

\newtheorem{introcorollary}[introtheorem]{Corollary}

\titleformat{\chapter}[block]{\Huge\scshape\bf\centering}{\thechapter.}{1em}{} 
\titleformat{\section}[block]{\large\scshape\bf\centering}{\thesection.}{1em}{} 
\titleformat{\subsection}[block]{\scshape\bf\centering}{\thesubsection.}{1em}{} 

\def\A{\mathcal{A}}
\def\aa{\boldsymbol{a}}
\def\add{\operatorname{add}}
\def\Ann{\operatorname{ann}}
\def\bb{\boldsymbol{b}}
\def\C{\mathsf{C}}
\def\CC{\mathbb{C}}

\def\cone{\operatorname{cone}}
\def\db{\mathsf{D^b}}
\def\ds{\mathsf{D_{sg}}}
\def\e{\operatorname{e}}
\def\End{\operatorname{End}}
\def\Jac{\operatorname{jac}}
\def\k{\operatorname{K}}
\def\lten{\otimes^{\mathbf{L}}}
\def\mod{\operatorname{mod}}
\def\N{\mathbb{N}}
\def\n{\mathfrak{n}}
\def\smd{\operatorname{\mathsf{smd}}}
\def\T{\mathcal{T}}
\def\X{\mathcal{X}}
\def\xx{\boldsymbol{x}}
\def\Y{\mathcal{Y}}
\def\Z{\mathcal{Z}}
\def\ZZ{\mathbb{Z}}

\author[\"{O}.~Esentepe]{\"{O}zg\"{u}r Esentepe}
\address{School of Mathematics, University of Leeds, Leeds, UK}
\email{o.esentepe@leeds.ac.uk}

\author[R.~Takahashi]{Ryo Takahashi}
\address{Graduate School of Mathematics, Nagoya University, Furocho, Chikusaku, Nagoya, 464-8602, Japan}
\email{takahashi@math.nagoya-u.ac.jp}
\urladdr{https://www.math.nagoya-u.ac.jp/~takahashi/}

\title{Annihilators and decompositions of singularity categories}


\thanks{R.T. was partly supported by JSPS Grant-in-Aid for Scientific Research 23K03070. \"{O}.E. was supported by the Engineering and Physical Sciences Research Council [grant number EP/W007509/1].}

\begin{document}

\begin{abstract}
Given any commutative Noetherian ring $R$ and an element $x$ in $R$, we consider the full subcategory $\C(x)$ of its singularity category consisting of objects for which the morphism that is given by the  multiplication by $x$ is zero. Our main observation is that we can establish a relation between $\C(x), \C(y)$ and $\C(xy)$ for any two ring elements $x$ and $y$. Utilizing this observation, we obtain a decomposition of the singularity category and consequently an upper bound on the dimension of the singularity category. 
\end{abstract}

\maketitle

\thispagestyle{empty}

\section{Introduction}

 Hilbert's syzygy theorem is one of the most beautiful theorems in commutative algebra and algebraic geometry. It is considered to be the introduction of homological methods into this area. In today's language, Hilbert's theorem states that over a polynomial ring with $n$ variables over a field, any finitely generated module is quasi-isomorphic to a perfect complex of length at most $n$. During the first half of the twentieth century, the goal was to answer this question for other rings: over which rings can we say that \textit{any finitely generated module} is quasi-isomorphic to a perfect complex? And during the second half of the century, with the language and the rising importance of derived categories, the question became: over which rings can we say that \textit{any bounded complex of finitely generated modules} is quasi-isomorphic to a perfect complex? This question motivates the definition of the \textit{singularity category} (or the \textit{stabilized derived category}) of a Noetherian ring which is defined as the Verdier quotient of the bounded derived category of finitely generated modules by the (full) subcategory of perfect complexes. Note that this quotient vanishes if and only if the answer to the above question is affirmative. The name singularity category suggests that there is a geometric connection. Indeed, due to the work of Auslander, Buchsbaum and Serre, we know that the singularity category $\ds(R)$ of a commutative Noetherian ring $R$ vanishes if and only if $R$ is regular. Thus, it provides a measure of the singularities of the affine space $\spec(R)$.

The goal of this paper is to further our understanding of the structure of the singularity category of a commutative Noetherian ring $R$. Our approach uses the annihilators of the singularity category and our main observation is the following: for any $x,y \in R$, we have 
\begin{align}\label{main-observation}
    \C(xy) = \smd(\C(x) \ast \C(y)).
\end{align}
Let us explain the notation. For each $r\in R$ we denote by $\C(r)$ the subcategory of $\ds(R)$ consisting of objects $X$ such that the morphism $X\xrightarrow{r}X$ in $\ds(R)$ given by the multiplication by $r$ is zero. The subcategory $\C(x) \ast \C(y)$ consists of all objects $E$ that fits into an exact triangle $X \to E \to Y \to X[1] $ in $\ds(R)$ with $X \in \C(x)$ and $Y \in \C(y)$. Finally, by $\smd(\C(x) \ast \C(y))$, we denote the subcategory consisting of all direct summands of objects belonging to $\C(x) \ast \C(y)$. This observation is motivated by a result of Dugas and Leuschke \cite{DL} who proved the special case $\C(z^{m+n}) = \smd(\C(z^m) \ast \C(z^n))$ where $R =S / (f+z^k)$ is the $k$-fold branched cover of a hypersurface singularity $S/(f)$. We note that if $\C(xy) = \ds(R)$, then (\ref{main-observation}) gives a decomposition of the singularity category. This is the first part of our main theorem. In fact, using the associativity of the $\ast$ operator, we have the following.

\begin{introtheorem}
Let $x_1,\dots,x_n$ be elements of a commutative Noetherian ring $R$ such that the product $x_1\cdots x_n$ belongs to $\Ann\ds(R)$.
Then
$$
\ds(R)=\smd(\C(x_1)\ast\cdots\ast\C(x_n)).
$$
\end{introtheorem}

The second part of our main theorem concerns the dimension of the singularity category of a commutative Noetherian ring. Recall that the \emph{(Rouquier) dimension} $\dim\T$ of a triangulated category $\T$ is an invariant which measures how much it costs to build it from a single object using the ``free" operations of finite direct sums, direct summands and shift and the cone operation which costs ``1 unit" each time it is applied.

\begin{introtheorem}
Let $x_1,\dots,x_n$ be elements of a commutative Noetherian ring $R$ such that the product $x_1\cdots x_n$ belongs to $\Ann\ds(R)$. If no $x_i$ is a unit or a zerodivisor, then there is an inequality
$$
\dim\ds(R)\le\sum_{i=1}^n\dim\ds(R/x_iR)+n-1.
$$
\end{introtheorem}

In these theorems, $\Ann\ds(R)$ is the \emph{annihilator of the singularity category} which is defined as the ideal of $R$ consisting of elements $r$ such that $\C(r) = \ds(R)$. We prove these theorems in Section 2 after providing the necessary background and preliminaries. It is worth comparing them with similar results shown in \cite{dec,sgd}.

In Section 3, we consider the case of a sequence elements. After proving a weaker version of Theorem A, we prove the following.

\begin{introtheorem}
    Let $\xx = x_1,  \ldots, x_n$ a regular sequence on $R$ such that $x_1^{m_1}, \ldots, x_n^{m_n} \in  \Ann \ds(R)$ for some positive integers $m_1, \ldots, m_n$. Assume that $a_1, \cdots, a_n$ are nonnegative integers such that $x_1^{a_1}\cdots x_n^{a_n} \in \Ann \ds(R)$. Put $\omega = m_1 \cdots m_n \left( \dfrac{a_1}{m_1} + \cdots + \dfrac{a_n}{m_n} \right)$. Then, we have
\begin{align*}
    \dim \ds(R) \leq \omega \left( \dim \ds(R/\xx R) + 1 \right) - 1.
\end{align*}
In particular, $\dim \ds(R) \leq m\,( \dim \ds(R/\xx R) + 1) - 1$, where $m=m_1\cdots m_n$.
\end{introtheorem}

We dedicate Section 4 to examples and applications of these theorems. We compute several examples and discuss the case of isolated singularities: When $\dim \ds(R)$ is finite, the vanishing locus of $\Ann \ds(R)$ is the singular locus of $R$. In particular, if $(R, \m)$ is a local ring with an isolated singularity for which $\dim \ds(R) < \infty$, then $\Ann \ds(R)$ is $\m$-primary. In this case, for any $x \in \m$, we have $x^n \in \Ann \ds(R)$ for some positive integer $n$. If we define $\alpha(x) = \min\{ n \mid x^n \in \Ann \ds(R) \}$, then we have the following corollary.

\begin{introcorollary}
Let $(R, \m)$ be a commutative Noetherian local ring with an isolated singularity. If $\dim \ds(R) < \infty$, then 
\begin{align*}
    \dim \ds(R) \leq \inf\{\alpha (x)(\dim\ds(R/xR)+1) - 1\},
\end{align*}
where $x$ runs over the nonzerodivisors in $\m$.
In particular, if $\dim \ds(R/xR) = 0$ for some nonzerodivisor $x\in\m$, then $\dim \ds(R) \leq \alpha(x) - 1$.
\end{introcorollary}


\section{Proof of the main theorem}

In this section, we introduce and prove preliminaries about the subcategories which appear in our main observation \ref{main-observation}. Let us start with the conventions and the definitions that we will use.

\begin{convention}
Throughout the rest of this section, let $R$ be a commutative Noetherian ring. Denote by $\mod R$ the category of finitely generated $R$-modules.
Denote by $\db(R)$ the bounded derived category of $\mod R$.
Let $\ds(R)$ stand for the {\em singularity category} of $R$, which is by definition the Verdier quotient of $\db(R)$ by the perfect complexes. 
We assume that all $R$-modules are finitely generated and all subcategories are strictly full.
We may omit subscripts/superscripts unless there is a danger of confusion.
\end{convention}

\begin{definition}
\begin{enumerate}[(1)]
\item
For an element $x\in R$, we denote by $\C(x)$ the subcategory of $\ds(R)$ consisting of objects $X$ such that the morphism $X\xrightarrow{x}X$ in $\ds(R)$ given by the multiplication by $x$ is zero.
Note that $\C(x)$ is closed under finite direct sums, direct summands, and shifts.
\item
Let $\A$ be an additive category.
For a subcategory $\X$ of $\A$, we denote by $\smd_\A\X$ the subcategory of $\A$ consisting of direct summands of objects in $\X$.
\item
Let $\T$ be a triangulated category.
For subcategories $\X,\Y$ of $\T$, we denote by $\X\ast\Y$ the subcategory of $\T$ consisting of objects $E\in\T$ such that there exists an exact triangle $X\to E\to Y\to X[1]$ in $\T$ with $X\in\X$ and $Y\in\Y$.
Using the octahedral axiom, one easily sees that $-\ast-$ is associative:
the equality $(\X\ast\Y)\ast\Z=\X\ast(\Y\ast\Z)$ holds for subcategories $\X,\Y,\Z$ of $\T$.
Hence there is no confusion in writing $\X_1\ast\cdots\ast\X_n$ for subcategories $\X_1,\dots,\X_n$ of $\T$.
\item
For a sequence $\xx=x_1,\dots,x_n$ of elements of $R$, we denote by $\k(\xx)$ the Koszul complex of $\xx$ over $R$.
\end{enumerate}
\end{definition}

\begin{remark}
Let $T\in\db(R)$.
In general, taking derived tensor product $T\lten_R-$ does not preserve an isomorphism in $\ds(R)$, but it does if $T$ is isomorphic to a perfect complex.
To be precise, let $P$ be a perfect complex over $R$.
Let $(X\xleftarrow{s}Z\xrightarrow{f}Y)$ be a morphism in $\ds(R)$, that is, $s$ and $f$ are morphisms in $\db(R)$ such that the mapping cone $\cone(s)$ of $s$ is isomorphic to a perfect complex.
Then $(P\lten_RX\xleftarrow{P\lten_Rs}P\lten_RZ\xrightarrow{P\lten_Rf}P\lten_RY)$ is a morphism in $\ds(R)$, since $\cone(P\lten_Rs)$ is isomorphic to a perfect complex.
One sees that the functor $P\lten_R-:\db(R)\to\db(R)$ induces a functor $\ds(R)\to\ds(R)$.
In particular, the implication
$$
\text{$X\cong Y$ in $\ds(R)$}\implies\text{$\k(\xx)\otimes X\cong\k(\xx)\otimes Y$ in $\ds(R)$}
$$
holds for any objects $X,Y$ of $\ds(R)$ and any sequence $\xx=x_1,\dots,x_n$ in $R$.
\end{remark}

The first part of the following lemma shows us the role that the Koszul complex plays in the rest of the section.

\begin{lemma}\label{4}
\begin{enumerate}[\rm(1)]
\item
Let $x\in R$.
Then $\C(x)=\smd\{\k(x)\otimes X\mid X\in\ds(R)\}$.
\item
Let $x,y\in R$.
Let $A\to B\to C\to A[1]$ be an exact triangle in $\ds(R)$.
If $A\in\C(x)$ and $C\in\C(y)$, then $B\in\C(xy)$.
\end{enumerate}
\end{lemma}

\begin{proof}
(1) 
The chain map $\k(x)\xrightarrow{x}\k(x)$ given by multiplication by $x$ is null-homotopic by \cite[Proposition 1.6.5(a)]{BH}, which shows the inclusion $(\supseteq)$.
To show the inclusion $(\subseteq)$, pick an object $Y\in\C(x)$.
There is an exact triangle $R\xrightarrow{x}R\to\k(x)\to R[1]$ in $\db(R)$.
Tensoring the complex $Y$, we get an exact triangle
\begin{equation}\label{3}
Y\xrightarrow{x}Y\to\k(x)\otimes Y\to Y[1]
\end{equation}
in $\db(R)$.
Since $(Y\xrightarrow{x}Y)=0$ in $\ds(R)$, the image of \eqref{3} in $\ds(R)$ yields an isomorphism\begin{equation}\label{504}
\k(x)\otimes Y\cong Y\oplus Y[1]
\end{equation} in $\ds(R)$.

(2) The argument given in \cite[Remark 2.6]{dec} shows the assertion.
\end{proof}

Before proving our main observation, we need one more technical lemma.

\begin{lemma}\label{2}
Let $\T$ be a triangulated category.
Let $\X,\Y$ be subcategories of $\T$.
Then there are equalities
$$
\smd((\smd\X)\ast\Y)=\smd(\X\ast\Y)=\smd(\X\ast(\smd\Y)).
$$
\end{lemma}

\begin{proof}
Since $\smd\X$ contains $\X$, we have that $\smd((\smd\X)\ast\Y)$ contains $\smd(\X\ast\Y)$.
To show the opposite inclusion, pick an object $A\in\smd((\smd\X)\ast\Y)$.
Then there exists an object $B\in\T$ such that $A\oplus B\in(\smd\X)\ast\Y$, which gives an exact triangle
\begin{equation}\label{1}
C\to A\oplus B\to Y\to C[1]
\end{equation}
in $\T$ with $C\in\smd\X$ and $Y\in\Y$.
There exists an object $D\in\T$ such that $C\oplus D\in\X$.
Taking the direct sum of \eqref{1} with the exact triangle $D\to D\to0\to D[1]$, we get an exact triangle
$$
C\oplus D\to A\oplus B\oplus D\to Y\to(C\oplus D)[1].
$$
Hence $A\oplus B\oplus D$ is in $\X\ast\Y$, and it follows that $A$ belongs to $\smd(\X\ast\Y)$.
The first equality of the lemma is now obtained, and the second equality is shown similarly.
\end{proof}

Now we are ready to prove our main observation. We note once again that this observation works in the generality of all commutative Noetherian rings.

\begin{proposition}\label{5}
Let $x_1,\dots,x_n\in R$.
Then there is an equality
$$
\C(x_1\cdots x_n)=\smd(\C(x_1)\ast\cdots\ast\C(x_n)).
$$
\end{proposition}

\begin{proof}
In view of Lemma \ref{2}, it suffices to show that $\C(xy)=\smd(\C(x)\ast\C(y))$ for $x,y\in R$.
Indeed, if we have done it and if $\C(x_1\cdots x_{n-1})=\smd(\C(x_1)\ast\cdots\ast\C(x_{n-1}))$, then we will have
\begin{multline*}
\C(x_1\cdots x_n)
=\C((x_1\cdots x_{n-1})x_n)
=\smd(\C(x_1\cdots x_{n-1})\ast\C(x_n))\\
=\smd(\smd(\C(x_1)\ast\cdots\ast\C(x_{n-1}))\ast\C(x_n))
=\smd(\C(x_1)\ast\cdots\ast\C(x_{n-1})\ast\C(x_n)).
\end{multline*}

Let us show the inclusion $\C(xy)\subseteq\smd(\C(x)\ast\C(y))$.
Pick an object $A\in\C(xy)$.
Applying the octahedral axiom, we get a commutative diagram
$$
\xymatrix{
A\ar[r]^x\ar@{=}[d]& A\ar[r]\ar[d]^y& \k(x)\otimes A\ar[r]\ar[d]& A[1]\ar@{=}[d]\\
A\ar[r]^{xy}_0\ar[d]& A\ar[r]\ar@{=}[d]& B\ar[r]\ar[d]& A[1]\ar[d]\\
A\ar[r]^y\ar[d]& A\ar[r]\ar[d]& \k(y)\otimes A\ar[r]\ar@{=}[d]& A[1]\ar[d]\\
\k(x)\otimes A\ar[r]& B\ar[r]& \k(y)\otimes A\ar[r]& (\k(x)\otimes A)[1]
}
$$
in $\ds(R)$ whose rows are exact triangles.
The second row implies $B\cong A\oplus A[1]$, while the bottom row shows $B\in\C(x)\ast\C(y)$ by Lemma \ref{4}(1).
Therefore, $A$ belongs to $\smd(\C(x)\ast\C(y))$.

Let us show the inclusion $\C(xy)\supseteq\smd(\C(x)\ast\C(y))$.
Pick $A\in\smd(\C(x)\ast\C(y))$.
Then $A\oplus B\in\C(x)\ast\C(y)$ for some $B\in\ds(R)$, and there exists an exact triangle $C\to A\oplus B\to D\to C[1]$ in $\ds(R)$ such that $C\in\C(x)$ and $D\in\C(y)$.
Lemma \ref{4}(2) says that $A\oplus B$ is in $\C(xy)$, and so is its direct summand $A$.
\end{proof}

We are interested in the situation where the left hand side of the equality in the preceding proposition is the entire singularity category. That is, we would like to consider the case where for any object in the bounded derived category, multiplication by $x_1 \cdots x_n$ factors through a perfect complex. Let us make this more precise and introduce notation with the following definition.

\begin{definition}
We introduce two kinds of annihilators as follows.
\begin{enumerate}[(1)]
\item
The {\em annihilator} of an object $X\in\ds(R)$ is defined by:
\begin{align*}
\Ann_{\ds(R)}X
&=\{a\in R\mid(X\xrightarrow{a}X)=0\text{ in }\ds(R)\}\\
&=\{a\in R\mid X\in\C(a)\}=\Ann_R(\End_{\ds(R)}(X)).
\end{align*}
\item
The {\em annihilator} $\Ann\ds(R)$ of the category $\ds(R)$ is defined by:
\begin{align*}
\Ann\ds(R)&=\{a\in R\mid(X\xrightarrow{a}X)=0\text{ for all }X\in\ds(R)\}\\
&=\{a\in R\mid\C(a)=\ds(R)\}
=\bigcap_{X\in\ds(R)}\Ann_{\ds(R)}X.
\end{align*}
\end{enumerate}
Note that both of the annihilators $\Ann_{\ds(R)}X$ and $\Ann\ds(R)$ are ideals of $R$.
\end{definition}

Now we recall the definition of the dimension of a triangulated category.

\begin{definition}
Let $\T$ be a triangulated category.
\begin{enumerate}[(1)]
\item
Let $\X$ be a subcategory of $\T$.
We denote by $\langle \X\rangle$ the additive closure $$\add\{X[i]\mid X\in\X,i\in\ZZ\}$$ in $\T$, that is, the smallest subcategory of $\T$ containing $\X$ and closed under finite direct sums, direct summands and shifts. 
We set $\langle\X\rangle_0^{\T}=0$ and define $\langle \X \rangle_n^\T = \langle\langle \X \rangle_{n-1}^\T \ast\langle\X\rangle\rangle$ by induction for any $n \geq 1$. Note that $\langle\X\rangle_1=\langle\X\rangle$.
When $\X$ consists of a single object $T$, we simply write $\langle T\rangle_n$ instead of $\langle\X\rangle_n$.
\item
The {\em (Rouquier) dimension} of $\T$ is defined by
$$
\dim\T=\inf\{n\ge0\mid\text{$\langle G\rangle_{n+1}=\T$ for some $G\in\T$}\}.
$$
\end{enumerate}
\end{definition}

We are now ready to give the main application of our main observation.

\begin{theorem}\label{7}
Let $x_1,\dots,x_n$ be elements of $R$ such that the product $x_1\cdots x_n$ belongs to $\Ann\ds(R)$.
Then
$$
\ds(R)=\smd(\C(x_1)\ast\cdots\ast\C(x_n)).
$$
If no $x_i$ is a unit or a zerodivisor, then there is an inequality
$$
\dim\ds(R)\le\sum_{i=1}^n\dim\ds(R/x_iR)+n-1.
$$
\end{theorem}

\begin{proof}
It immediately follows from Proposition \ref{5} that $\ds(R)=\smd(\C(x_1)\ast\cdots\ast\C(x_n))$.
Let us show the dimension inequality.
Since each $x_i$ is a non-unit, $R/x_iR$ is not a zero ring.
We may assume $\dim\ds(R/x_iR)=:d_i<\infty$ for each $i$.
There exists an object $G_i\in\ds(R/x_iR)$ such that $\ds(R/x_iR)=\langle G_i\rangle_{d_i+1}^{\ds(R/x_iR)}$.
Since $x_i$ is a non-zerodivisor, we have
$$
\C(x_i)=\smd\{\k(x_i)\otimes X\mid X\in\ds(R)\}=\smd\{R/x_iR\lten_RX\mid X\in\ds(R)\}
$$
by Lemma \ref{4}(1), and
\begin{equation}\label{6}
R/x_iR\lten_RX\in\ds(R/x_iR)=\langle G_i\rangle_{d_i+1}^{\ds(R/x_iR)}.
\end{equation}
As a perfect $R/x_iR$-complex is quasi-isomorphic to a perfect $R$-complex, the natural surjection $R\to R/x_iR$ induces an exact functor $\ds(R/x_iR)\to\ds(R)$.
Applying this functor to \eqref{6}, we observe that there is an inclusion $\C(x_i)\subseteq\langle G_i\rangle_{d_i+1}^{\ds(R)}$.
It follows that
\begin{align*}
\ds(R)=\smd(\C(x_1)\ast\cdots\ast\C(x_n))&\subseteq\smd(\langle G_1\rangle_{d_1+1}\ast\cdots\ast\langle G_n\rangle_{d_n+1})\\
&\subseteq\langle G_1\oplus\cdots\oplus G_n\rangle_{(d_1+1)+\cdots+(d_n+1)}\subseteq\ds(R).
\end{align*}
Hence the equality $\ds(R)=\langle G_1\oplus\cdots\oplus G_n\rangle_{(d_1+\cdots+d_n)+n}$ holds, which yields the inequality $\dim\ds(R)\le(d_1+\cdots+d_n)+n-1$.
\end{proof}

We should compare the above theorem with \cite[Theorem 1.1 and Corollary 2.12]{dec}, which have similar flavors.

\section{The case of a sequence of elements}

In the previous section, we considered the full subcategory of objects in the singularity category of a commutative Noetherian ring that are annihilated by a given ring element. In this section, we are going to further investigate such subcategories by considering their intersections. That is, we will consider the full subcategory of objects that are annihilated by a given sequence of ring elements.

\begin{definition}
\begin{enumerate}[(1)]
\item
For a sequence $\xx=x_1,\dots,x_n$ of elements of $R$, set $\C(\xx)=\bigcap_{i=1}^n\C(x_i)$.
Namely, $\C(\xx)$ is the subcategory of $\ds(R)$ consisting of all objects $X$ with $(X\xrightarrow{x_i}X)=0$ in $\ds(R)$ for all $i$.
\item
For an ideal $I$ of $R$, we set $\C(I)=\bigcap_{a\in I}\C(a)$.
This is the subcategory of $\ds(R)$ consisting of those objects $X$ which satisfy $(X\xrightarrow{a}X)=0$ in $\ds(R)$ for all $a\in I$.
Note that for any ideals $I,J$ of $R$, if $I\subseteq J$, then $\C(I)\supseteq\C(J)$.
\end{enumerate}
\end{definition}

The second assertion of the following proposition is a generalization of Lemma \ref{4}(1).

\begin{proposition}\label{11}
\begin{enumerate}[\rm(1)]
\item
Let $I$ be an ideal of $R$.
Let $\xx=x_1,\dots,x_n$ be a system of generators of $I$.
Then $\C(I)=\C(\xx)$.
In particular, the implication
$$
(\aa)=(\bb)\implies\C(\aa)=\C(\bb)
$$
holds for sequences of elements $\aa=a_1,\dots,a_r$ and $\bb=b_1,\dots,b_s$ in $R$.
\item
Let $\xx=x_1,\dots,x_n$ be a sequence of elements of $R$.
Then every object $C\in\C(\xx)$ satisfies
$$
\k(\xx)\otimes C\cong\bigoplus_{i=0}^n(C[i])^{\oplus\binom{n}{i}}
$$
in $\ds(R)$.
In particular, it holds that $\C(\xx)=\smd\{\k(\xx)\otimes X\mid X\in\ds(R)\}$.
\end{enumerate}
\end{proposition}

\begin{proof}
(1) It is clear that $\C(I)$ is contained in $\C(\xx)$.
To show the opposite inclusion, let $X\in\C(\xx)$.
Take any $a\in I$.
Then $a=\sum_{i=1}^n{a_ix_i}$ for some $a_1,\dots,a_n\in R$.
Note that $x_1,\dots,x_n$ are in $\ann_{\ds(R)}X$.
Since $\ann_{\ds(R)}X$ is an ideal of $R$, the element $a$ is also in $\ann_{\ds(R)}X$.
Hence $X$ belongs to $\C(I)$.

(2) The last assertion follows from the first assertion and the fact by Lemma \ref{4}(1) that
$$
\k(\xx)\otimes X\cong\k(x_i)\otimes(\k(x_1,\dots,x_{i-1},x_{i+1},\dots,x_n)\otimes X)\in\C(x_i)
$$
for each $1\le i\le n$.
To prove the first assertion, we use induction on $n$.
The case $n=0$ is obvious.
Let $n>0$.
The induction hypothesis gives rise to an isomorphism $\k(x_1,\dots,x_{n-1})\otimes C\cong\bigoplus_{i=0}^{n-1}C[i]^{\oplus\binom{n-1}{i}}$ in $\ds(R)$.
Tensoring with $\k(x_n)$, we get isomorphisms
\begin{align*}
\k(\xx)\otimes C
&\cong\k(x_n)\otimes(\k(x_1,\dots,x_{n-1})\otimes C)
\cong\k(x_n)\otimes(\textstyle\bigoplus_{i=0}^{n-1}C[i]^{\oplus\binom{n-1}{i}})\\
&\cong\textstyle\bigoplus_{i=0}^{n-1}(\k(x_n)\otimes C)[i]^{\oplus\binom{n-1}{i}}
\cong\textstyle\bigoplus_{i=0}^{n-1}(C\oplus C[1])[i]^{\oplus\binom{n-1}{i}}\\
&\cong\textstyle\bigoplus_{i=0}^{n-1}(C[i]\oplus C[i+1])^{\oplus\binom{n-1}{i}}
\cong\textstyle\bigoplus_{i=0}^{n-1}(C[i]^{\oplus\binom{n-1}{i}}\oplus C[i+1]^{\oplus\binom{n-1}{i}})\\
&\cong C\oplus(\textstyle\bigoplus_{i=1}^{n-1}C[i]^{\oplus\binom{n-1}{i}})\oplus(\textstyle\bigoplus_{i=0}^{n-2}C[i+1]^{\binom{n-1}{i}})\oplus C[n]\\
&\cong C\oplus\textstyle\bigoplus_{i=1}^{n-1}C[i]^{\oplus\binom{n-1}{i}+\binom{n-1}{i-1}}\oplus C[n]
\cong\textstyle\bigoplus_{i=0}^n(C[i])^{\oplus\binom{n}{i}},
\end{align*}
where the fourth isomorphism follows from \eqref{504}, and the last isomorphism holds since $\binom{n-1}{i}+\binom{n-1}{i-1}=\binom{n}{i}$.
\end{proof}

The following proposition is a weaker version of Proposition \ref{5} for this setting.

\begin{proposition}\label{multiplication-proposition}
For any elements $x_1,\dots,x_n,y,z\in R$ with $n\ge0$ one has:
$$
\C(x_1,\dots,x_n,yz)\subseteq\smd(\C(x_1,\dots,x_n,y)\ast\C(x_1,\dots,x_n,z)).
$$
\end{proposition}

\begin{proof}
Put $\xx=x_1,\dots,x_n$.
Let $X\in\C(\xx,yz)$.
Applying the octahedral axiom, we get a commutative diagram
$$
\xymatrix{
R\ar[r]^y\ar@{=}[d]& R\ar[r]\ar[d]^z& \k(y)\ar[r]\ar[d]& R[1]\ar@{=}[d]\\
R\ar[r]^{yz}\ar[d]& R\ar[r]\ar@{=}[d]& \k(yz)\ar[r]\ar[d]& R[1]\ar[d]\\
R\ar[r]^z\ar[d]& R\ar[r]\ar[d]& \k(z)\ar[r]\ar@{=}[d]& R[1]\ar[d]\\
\k(y)\ar[r]& \k(yz)\ar[r]& \k(z)\ar[r]& \k(y)[1]
}$$
in $\db(R)$ whose rows are exact triangles.
Tensoring $\k(\xx)\otimes X$ with the bottom row induces an exact triangle 
\begin{equation}\label{341}
\k(\xx,y)\otimes X\to\k(\xx,yz)\otimes X\to\k(\xx,z)\otimes X\to(\k(\xx,y)\otimes X)[1].
\end{equation}
It follows from Proposition \ref{11}(2) that $X$ is a direct summand in $\ds(R)$ of $\k(\xx,yz)\otimes X$, and the latter object belongs to $\C(\xx,y)\ast\C(\xx,z)$ by \eqref{341}.
Therefore, $X$ is in $\smd(\C(\xx,y)\ast\C(\xx,z))$.
\end{proof}

The following proposition is an immediate consequence of Proposition \ref{multiplication-proposition}.

\begin{lemma}\label{product-of-powers}
    Let $x_1, \dots, x_n \in R$ and assume that $m_1,\dots, m_n$ are nonnegative integers. Then, we have 
    \begin{align*}
        \C(x_1^{m_1}, \dots, x_n^{m_n}) \subseteq \langle \C(x_1, \dots, x_n) \rangle_m 
    \end{align*}
    where $m = m_1 \cdots  m_n$.
\end{lemma}

\begin{proof}
If $m_i=0$ for some $1\le i\le n$, then $\C(x_1^{m_1},\dots,x_n^{m_n})=\bigcap_{j=1}^n\C(x_j^{m_j})\subseteq\C(x_i^{m_i})=\C(1)=0$, and the assertion obviously holds. Assume that $m_1,\dots,m_n$ are all positive.
By induction on $m_n$, it follows from Proposition \ref{multiplication-proposition} and Lemma \ref{2} that we should have $\C(x_1^{m_1}, \dots, x_n^{m_n}) \subseteq  \langle \C(x_1^{m_1}, \dots, x_{n-1}^{m_{n-1}}, x_n) \rangle_{m_n}$. Then, the proof is finished by induction on $n$.
\end{proof}

As an application of Lemma \ref{product-of-powers}, we get the following.

\begin{lemma}\label{lemma-with-omega}
Let $x_1, \dots, x_n \in R$ and assume that $m_1, \dots, m_n$ are such that $x_1^{m_1}, \dots, x_n^{m_n} \in \Ann \ds(R)$. Let $m = m_1\cdots m_n$ and put $\omega_i = \dfrac{m}{m_i}a_i$ where $a_i$ is a nonnegative integer and $1 \leq i \leq n$. Then, we have
$$
\C(x_i^{a_i}) \subseteq \langle \C(x_1, \dots, x_n) \rangle_{\omega_i}.
$$
\end{lemma}

\begin{proof}
Note that for any $1 \leq j \leq n$, we have $\C(x_j^{m_j}) = \ds(R)$ as $x_j^{m_j} \in \Ann \ds(R)$. Therefore, we have $\C(x_i^{a_i}) = \C(x_1^{m_1}, \dots, x_{i-1}^{m_{i-1}}, x_i^{a_i}, x_{i+1}^{m_{i+1}}, \dots, x_n^{m_n})$. Then, the result follows from Lemma \ref{product-of-powers}.
\end{proof}

Combining our results, we have the following proposition.

\begin{proposition}\label{last-proposition-before-the-new-main-theorem}
Let $x_1, \ldots, x_n \in R$ and assume that $a_1, \ldots, a_n$ are nonnegative integers. Let $\omega_1,\dots,\omega_n$ be as in Lemma \ref{lemma-with-omega}. Then, we have
\begin{align*}
    \C(x_1^{a_1} \cdots x_n^{a_n}) \subseteq \langle \C(x_1, \dots, x_n)\rangle_{\omega_1 + \cdots + \omega_n}. 
\end{align*}
In particular, if $a_1, \ldots, a_n$ are such that $x_1^{a_1} \cdots x_n^{a_n} \in \Ann \ds(R)$, then we have
\begin{align*}
    \ds(R) = \langle \C(x_1, \dots, x_n)\rangle_{\omega_1 + \cdots + \omega_n}.
\end{align*}
\end{proposition}

\begin{proof}
Applying Proposition \ref{5}, we have $\C(x_1^{a_1}\cdots x_n^{a_n})=\smd(\C(x_1^{a_1})\ast\cdots\ast\C(x_n^{a_n}))$.
It follows from Lemma \ref{lemma-with-omega} that $\C(x_i^{a_i})\subseteq\langle\C(x_1,\dots,x_n)\rangle_{\omega_i}$ for each $1\le i\le n$.
Therefore, $\C(x_1^{a_1} \cdots x_n^{a_n}) \subseteq \langle \C(x_1, \dots, x_n)\rangle_{\omega_1 + \cdots + \omega_n}$. 
\end{proof}

We are now ready to state our main theorem in this section.
We should compare it with Theorem \ref{7}.

\begin{theorem}\label{new-version-of-theorem}
Let $\xx = x_1,  \ldots, x_n$ a regular sequence on $R$ such that $x_1^{m_1}, \ldots, x_n^{m_n} \in  \Ann \ds(R)$ for some positive integers $m_1, \ldots, m_n$. Assume that $a_1, \cdots, a_n$ are nonnegative integers such that $x_1^{a_1}\cdots x_n^{a_n} \in \Ann \ds(R)$.
\begin{enumerate}[\rm(1)]
\item
There is an inequality
$$
\dim \ds(R) \leq \omega \left( \dim \ds(R/\xx R) + 1 \right) - 1,
$$
where $\omega = m_1 \cdots m_n(\frac{a_1}{m_1} + \cdots + \frac{a_n}{m_n})$.
\item
There is an inequality
$$
\dim \ds(R) \leq m\,( \dim \ds(R/\xx R) + 1) - 1,
$$
where $m=m_1\cdots m_n$.
\end{enumerate}
\end{theorem}

\begin{proof}
The second assertion follows from the first by letting $a_1=m_1$ and $a_2=\cdots=a_n=0$. In the following, we show the first assertion.
The proof is similar to the proof of Theorem \ref{7}; we use Proposition \ref{11}(2) instead of Lemma \ref{4}(1).
    Let us assume that $\dim\ds(R/\xx R) = d < \infty$. Then, there exists an object $G \in \ds(R/\xx R)$ such that $\ds(R/\xx R) = \langle G \rangle_{d+1}^{\ds(R/\xx R)}$. Since $\xx$ is a regular sequence, we have
    \begin{align*}
        \C(\xx) = \smd\{ \k(\xx) \otimes X \mid X \in \ds(R) \}= \smd\{R/\xx R\lten_RX\mid X\in\ds(R)\}
    \end{align*}
    and
    \begin{align*}
        R/\xx R\lten_RX\in\ds(R/\xx R)=\langle G \rangle_{d+1}^{\ds(R/\xx R)}.
    \end{align*}
    By applying the exact functor $\ds(R/ \xx R) \to \ds(R)$ induced by the natural projection $R \to R/\xx R$, we see that there is an inclusion $\C(\xx) \subseteq \langle G \rangle_{d+1}^{\ds(R)}$. Combining this with Proposition \ref{last-proposition-before-the-new-main-theorem}, we see that $\ds(R) = \langle G \rangle_{\omega(d+1)}$ which finishes the proof.
\end{proof}

\section{Applications} 

In this section, we provide some applications of Theorems \ref{7} and \ref{new-version-of-theorem}.
These two results prove useful when the singularity categories of the quotients on the right hand side are in fact of dimension zero.
We begin with recalling some definitions.

\begin{definition}
\begin{enumerate}[(1)]
\item
Suppose that $R$ is a Gorenstein local ring.
By $\MCM(R)$ we denote the category of maximal Cohen--Macaulay modules and by $\sMCM(R)$ the corresponding stable category.
The celebrated theorem due to Buchweitz \cite[Theorem 4.4.1]{B} asserts that the assignment $M\mapsto M$ gives a triangle equivalence
$$
\Delta:\underline{\MCM}(R)\xrightarrow{\cong}\ds(R).
$$
\item
Recall that the \textit{cohomology annihilator ideal} $\ca(R)$ of a commutative Noetherian ring $R$ is defined as 
\begin{align*}
    \ca(R) = \bigcup_{n =0}^\infty \{ r \in R \mid r \in \ann_R \Ext_R^i(M,N) \text{ for any } i \geq n \text{ and } M, N \in \module R \}.
\end{align*}
When $R$ is a Gorenstein local ring, it follows from \cite{E} that
$$
\ca(R)=\ann\ds(R).
$$
More precisely, it is equal to the ideal consisting of those $r \in R$ which annihilate every Hom-set in the stable category of maximal Cohen--Macaulay modules.
\end{enumerate}
\end{definition}

\begin{corollary}\label{8}
Let $(R,\m)$ be a Gorenstein local ring.
Let $x_1,\dots,x_n\in\m$ be non-zerodivisors such that the product $x_1\cdots x_n$ belongs to the cohomology annihilator $\ca(R)$.
Suppose that each $R/x_iR$ has finite CM-representation type.
Then $\dim\underline{\MCM}(R)\le n-1$.
\end{corollary}

\begin{proof}
There is an equality $\ca(R)=\Ann\ds(R)$.
Since each $R/x_iR$ has finite CM-representation type, we have $\dim\ds(R/x_iR)=\dim\underline{\MCM}(R/x_iR)\le0$.
The assertion now follows from Theorem \ref{7}.
\end{proof}

Let $(S,\n)$ be a regular local ring, and let $R=S/(f)$, where $0\ne f\in\n^2$.
We say that $R$ is a {\em simple singularity} if there exist only finitely many ideals $I$ of $R$ such that $f\in I^2$. When $R=k[\![x_1,\dots,x_n]\!]/(f_1,\dots,f_m)$ with $k$ a field, we denote by $\Jac R$ the {\em Jacobian ideal} of $R$, which is defined as the ideal of $R$ generated by the $h$-minors of the Jacobian matrix $(\frac{\partial f_i}{\partial x_j})$, where $h$ is the height of the ideal $(f_1,\dots,f_m)$ of the formal power series ring $k[\![x_1,\dots,x_n]\!]$.

\begin{corollary}\label{9}
Let $d\ge1$ and $e\ge2$ be integers.
Let $k$ be an algebraically closed field whose characteristic is neither $2,3,5$ nor divides $e$.
Let $f\in(x_1,\dots,x_d)^2k[\![x_1,\dots,x_d]\!]$ be such that $k[\![x_1,\dots,x_d]\!]/(f)$ is a simple singularity.
Let $R=k[\![x_0,\dots,x_d]\!]/(x_0^e+f)$.
Then $\dim\underline{\MCM}(R)\le e-2$.
\end{corollary}

\begin{proof}
We have $x_0^{e-1}=e^{-1}(ex_0^{e-1})\in\Jac R\subseteq\ca(R)$ by \cite[Example 2.7]{ua}.
Also, $R/x_0R=k[\![x_1,\dots,x_d]\!]/(f)$ is a simple singularity, so that it has finite CM-representation type by \cite[Theorem 9.8]{LW}.
By virtue of Corollary \ref{8}, we get $\dim\underline{\MCM}(R)\le(e-1)-1=e-2$.
\end{proof}

Let $(R,\m)$ be a local ring.
We denote by $\ell(R)$ the length of $R$ as an $R$-module, and by $\ell\ell(R)$ the {\em Loewy length} of $R$, that is, the infimum of integers $n\ge0$ such that $\m^n=0$.
For an $\m$-primary ideal $I$ of $R$, we denote by $\e(I)$ the {\em (Hilbert--Samuel) multiplicity} of $I$, i.e., $\e(I)=\lim_{n\to\infty}\frac{d!}{n^d}\ell(R/I^{n+1})\in\N$, where $d=\dim R$.

\begin{example}\label{0}
Let $R=\CC[\![x,y]\!]/(x^a-y^b)$ with $2\le a\le b$.
Then it is easy to observe that $\CC[\![y]\!]/(y^b)$ is a simple singularity.
Corollary \ref{9} gives rise to the inequality
\begin{equation}\label{102}
\dim\underline{\MCM}(R)\le a-2.
\end{equation}
Let $J=\Jac R$ be the Jacobian ideal of $R$.
Then we have $J=(x^{a-1},y^{b-1})$ and $R/J=\CC[\![x,y]\!]/(x^{a-1},y^{b-1})$.
We easily see that $\ell\ell(R/J)=(a-1)+(b-1)-1=a+b-3$.
It follows from \cite[Proposition 4.11]{BFK} or \cite[Corollary 1.3(1)]{sing} that
$$
\dim\underline{\MCM}(R)\le2(a+b-3)-1=2(a+b)-7=(a-2)+(a+2b-5).
$$
Take the parameter ideal $Q=(x^{a-1})$ of $R$ contained in $J$.
We claim that $J^a=QJ^{a-1}$.
In fact, we have $J^a=(x^{a-1},y^{b-1})^a=(\{(x^{a-1})^i(y^{b-1})^{a-i}\}_{i=1}^a)+(y^{(b-1)a})$, and
$$
(x^{a-1})^i(y^{b-1})^{a-i}=(x^a)^{i-1}x^{a-i}y^{(b-1)(a-i)}=(y^b)^{i-1}x^{a-i}y^{(b-1)(a-i)}=x^{a-i}y^{(a-1)(b-1)}y^{i-1}.
$$
As $a\le b$, we have $y^{b-1}\in(y^{a-1})$.
There are equalities
\begin{align*}
J^a&=y^{(a-1)(b-1)}(\{x^{a-i}y^{i-1}\}_{i=1}^a,y^{b-1})=y^{(a-1)(b-1)}(x^{a-1},x^{a-2}y,\dots,xy^{a-2},y^{a-1},y^{b-1})\\
&=y^{(a-1)(b-1)}(x^{a-1},x^{a-2}y,\dots,xy^{a-2},y^{a-1})=y^{(a-1)(b-1)}(x,y)^{a-1}.
\end{align*}
Therefore, we get
\begin{align*}
QJ^{a-1}&=x^{a-1}(x^{a-1},y^{a-1})^{a-1}=(x(x^{a-1},y^{b-1}))^{a-1}=(x^a,xy^{b-1})^{a-1}\\
&=(y^b,xy^{b-1})^{a-1}=(y^{b-1}(y,x))^{a-1}=y^{(a-1)(b-1)}(x,y)^{a-1}=J^a.
\end{align*}
This claim says that the parameter ideal $Q$ is a reduction of $J$, and we obtain $\e(J)=\ell(R/Q)=\ell(\CC[\![x,y]\!]/(x^{a-1},y^b))=(a-1)b$.
It follows from \cite[Corollary 1.3(2)]{sing} that
$$
\dim\underline{\MCM}(R)\le(a-1)b-1=(a-2)+(a-1)(b-1).
$$
As $a,b$ are at least $2$, both of the integers $a+2b-5$ and $(a-1)(b-1)$ are positive.
Therefore, the upper bound \eqref{102} for the dimension of the triangulated category $\underline{\MCM}(R)$ produced by Corollary \ref{9} is better than the upper bounds produced by \cite{BFK,sing}.
Furthermore, we should notice that when $a=3$ and $b\ge6$, the ring $R$ is not of finite CM-representation type by \cite[Chapter 9]{Y}, so that \eqref{102} and \cite[Theorem 1.2]{M} imply $\dim\underline{\MCM}(R)=1$. This also says that the inequality \eqref{102} is the best possible.
\end{example}

We should also compare Corollaries \ref{8}, \ref{9} and Example \ref{0} with \cite[Corollaries 1.2 and 1.3]{sgd}, which are stated in a similar context.

We present one more example, applying Theorem \ref{new-version-of-theorem}.

\begin{example}
Let $R = \CC[\![x,y,z,w]\!]/(f)$, where $f=x^3+y^3+xyz+w^2$.
Then, the partial derivatives of $f$ with respect to $x$, $y$ and $z$ are $3x^2 + yz$, $3y^2+xz$ and $xy$, respectively.
The equalities $x^3=\dfrac{1}{3}x(3x^2+yz)-\dfrac{1}{3}z(xy)$ and $y^3=\dfrac{1}{3}y(3y^2+xz)-\dfrac{1}{3}z(xy)$ hold in $R$. 
It follows from \cite[Example 2.7]{ua} that $x^3$, $y^3$ and $xy$ are in $\ca(R)$.
On the other hand, note that $x,y$ is a regular sequence on $R$ and
$$
R / (x,y)R \cong \CC[\![z,w]\!]/(w^2)
$$
is a complete local hypersurface of countable CM-representation type by \cite[Proposition 4.1]{BGS} or \cite[Example (6.5)]{Y} and thus $\dim \sMCM(R / (x,y)R) = 1$ by \cite[Propositions 2.4 and 2.7]{radius}. Therefore, by Theorem \ref{new-version-of-theorem}, we conclude that 
$$
\dim \sMCM(R) \le 3\cdot3\cdot\left(\frac{1}{3}+\frac{1}{3}\right)\cdot(1+1)-1=11. 
$$
\end{example}

Our results on the dimension of singularity categories require the condition that the product of some nonzerodivisors belong to $\Ann \ds(R)$. Hence, it is reasonable to ask whether this condition is too strong or when it happens. Now we recall a recent result on the annihilation of singularity categories and formulate our results from this point of view in the case of isolated singularities.

\begin{theorem}\cite[Theorem 4.6]{liu}
Let $R$ be a commutative Noetherian ring. If $\dim \ds(R) < \infty$, then 
\begin{align*}
    \sing(R) = \mathrm{V}(\Ann\ds(R)).
\end{align*}
\end{theorem}

Therefore, the annihilator of the singularity is not only a homological invariant, but also plays a role geometrically. For instance, it tells us that when the dimension of the singularity category of a commutative Noetherian ring is finite, then the singular locus is a closed subset. 

We start with the finiteness condition on $\dim \ds(R)$ in Liu's theorem. To this end, let us consider the following three conditions on a commutative Noetherian local ring $(R, \m)$:
\begin{description}
\item[A] There exists a nonzerodivisor $x \in \m$ such that $\dim \ds(R/x R) < \infty$ and $x^i \in \Ann \ds(R)$ for some $i > 0$.
\item[B] $\dim \ds(R) < \infty$.
\item[C] For any $x \in \m$, there is a positive integer $\ell_x$ such that $x^{\ell_x} \in \Ann \ds(R)$. 
\end{description} 

Then, Theorem \ref{new-version-of-theorem} immediately tells us that \textbf{A} implies \textbf{B}. If, moreover, $R$ has only isolated singularities, then the above theorem due to Liu tells us that $\Ann \ds(R)$ is $\m$-primary, which tells us that \textbf{B} implies \textbf{C}. Note that the $\ell_x$ are bounded above by the Loewy length $\ell\ell(R/\Ann \ds(R))$ of $R/\Ann \ds(R)$.

For any nonzerodivisor $x \in \m$, let $\ell_x$ be as above and put $d_x = \dim \ds(R/xR)$. Then, we have the following corollary of Theorem \ref{new-version-of-theorem}.

\begin{corollary}
Let $(R,\m)$ be an isolated singularity with $\dim \ds(R) < \infty$. Then, for any nonzerodivisor $x \in \m$, we have
\begin{align*}
    \dim \ds(R) \leq  \ell_x(d_x+1) - 1\le\ell\ell(R/\Ann\ds(R))(\dim\ds(R/xR)+1)-1.
\end{align*}
\end{corollary}

\end{document}